\documentclass[10pt,oneside,reqno]{amsart}
\usepackage{amsfonts,amsmath,amsthm,enumerate,mathtools,tikz,hyperref}
\usepackage[margin=1.25in]{geometry}

\newcommand\supp{\mathop{\rm supp}}
\newcommand\tr{\mathop{\rm tr}}

\DeclareMathOperator{\sgn}{sgn}

\DeclareMathOperator{\Hom}{Hom}

\newcommand{\bbD}{{\mathbb{D}}}

\newcommand{\bbN}{{\mathbb{N}}}
\newcommand{\bbR}{{\mathbb{R}}}
\newcommand{\bbP}{{\mathbb{P}}}
\newcommand{\bbZ}{{\mathbb{Z}}}
\newcommand{\bbC}{{\mathbb{C}}}

\newcommand{\bbT}{{\mathbb{T}}}

\newcommand{\cA}{{\mathcal{A}}}
\newcommand{\cB}{{\mathcal{B}}}

\newcommand{\cD}{{\mathcal{D}}}

\newcommand{\cJ}{{\mathcal{J}}}

\newcommand{\cM}{{\mathcal{M}}}

\newcommand{\cP}{{\mathcal{P}}}
\newcommand{\cQ}{{\mathcal{Q}}}

\newcommand{\cV}{{\mathcal{V}}}

\newcommand{\z}{\zeta}

\renewcommand{\Re}{\operatorname{Re}}

\theoremstyle{plain}
\newtheorem{thm}{Theorem}[section]
\newtheorem*{thm*}{Theorem}
\newtheorem{lem}[thm]{Lemma}
\newtheorem{prop}[thm]{Proposition}

\theoremstyle{definition}
\newtheorem{defn}{Definition}[section]

\newtheorem*{rem}{Remark}

\numberwithin{equation}{section}
\numberwithin{thm}{section}
\numberwithin{defn}{section}

\begin{document}

\title{Almost Periodicity in Time of Solutions of the Toda Lattice}

\author{Ilia Binder}

\address{Department of Mathematics, University of Toronto, Bahen Centre, 40 St. George St., Toronto, Ontario, CANADA M5S 2E4}

\email{ilia@math.toronto.edu}

\thanks{I.\ B.\ was supported in part by an NSERC Discovery grant.}

\author{David Damanik}

\address{Department of Mathematics, Rice University, Houston, TX 77005, U.S.A.}

\email{damanik@rice.edu}

\thanks{D.\ D.\ and T.\ V.\ were supported in part by NSF grant DMS--1361625.}

\author{Milivoje Lukic}

\address{Department of Mathematics, Rice University, Houston TX 77005, U.S.A.}

\email{milivoje.lukic@rice.edu}

\thanks{M.\ L.\ was supported in part by NSF grant DMS--1301582.}

\author{Tom VandenBoom}

\address{Department of Mathematics, Rice University, Houston, TX 77005, U.S.A.}

\email{tvandenboom@rice.edu}

\date{}

\begin{abstract}
We study an initial value problem for the Toda lattice with almost periodic initial data.  We consider initial data for which the associated Jacobi operator is absolutely continuous and has a spectrum satisfying a Craig-type condition, and show the boundedness and almost periodicity in time and space of solutions.
\end{abstract}

\maketitle

\tableofcontents

\section{Introduction}

This paper is devoted to studying solutions of the Toda lattice
\begin{align}
\label{eqn:todalattice1}
\frac{d}{dt}a_n(t) &= a_n(t)(b_{n+1}(t) - b_n(t)), \\
\label{eqn:todalattice2}
\frac{d}{dt}b_n(t) &= 2(a^2_n(t) - a^2_{n-1}(t)),
\end{align}
satisfying an initial value condition
\begin{align}
\label{eqn:todalatticeIV}
(a_n,b_n)(0) = (\tilde{a}_n,\tilde{b}_n)
\end{align}
for all $n \in \bbZ$.  Whenever the initial data is bounded, i.e. $\tilde a, \tilde b \in \ell^\infty(\mathbb{Z})$, there is a unique solution in $C^\infty(\mathbb{R}, \ell^\infty(\mathbb{Z}) \times \ell^\infty(\mathbb{Z}))$ of the initial value problem \eqref{eqn:todalattice1}, \eqref{eqn:todalattice2}, \eqref{eqn:todalatticeIV} \cite[Theorem 12.6]{TES00}.


The Toda lattice was originally proposed by Morikazu Toda in 1967 \cite{TOD67} as a model describing the positions and momenta of a chain of $k$ particles with positions $\{p_n\}_{n=1}^k$ and momenta $\{q_n\}_{n=1}^k$ interacting with Hamiltonian
\begin{align}
\label{eqn:hamiltonian}
H = \frac{1}{2}\sum_{n=1}^kp_n^2 + \sum_{n=1}^{k-1} \exp(-(q_{n+1}-q_n)) + \exp(-(q_n- q_{n-1})).
\end{align}
This system generalizes naturally to the case of infinitely many particles.  Toda's lattice is significant for being the first Hamiltonian with nearest-neighbor interactions demonstrating the existence of soliton solutions.

Toda was not the first to explore solid models with nearest-neighbor interactions.  Notably, in numerical experiments, Fermi, Pasta, and Ulam observed a surprising characteristic of systems defined by Hamiltonians with polynomial nonlinear interaction terms: that, rather than thermalizing to an equilibrium state, the particles asymptotically tended to almost periodic solutions in time \cite{FPU55}.  We intend to explore this same phenomenon  for the Toda lattice \eqref{eqn:hamiltonian}.

The formulation of the system described by \eqref{eqn:hamiltonian} as \eqref{eqn:todalattice1}, \eqref{eqn:todalattice2} follows from the change of variables due to Flaschka \cite{FLA74} given by
\begin{align*}
a_n &= \frac{1}{2}\exp(-\frac{1}{2}(q_{n+1}-q_n)), \\
b_n &= -\frac{1}{2}p_n.
\end{align*}
Flaschka's variables demonstrate that the Toda lattice is the discrete analogue of the Korteweg-de Vries equation \cite{KdV95} in the sense that the system \eqref{eqn:todalattice1}, \eqref{eqn:todalattice2} can be expressed equivalently as a Lax pair \cite{LAX68}
\begin{align}
\label{eqn:laxpair}
\frac{d}{dt}J(t) = [P(t),J(t)]
\end{align}
where
\begin{align*}
(J(t)u)_n = a_{n-1}(t)u_{n-1} + b_n(t)u_n + a_n(t)u_{n+1}
\end{align*}
and
\begin{align*}
(P(t)u)_n = -a_{n-1}(t)u_{n-1} + a_n(t)u_{n+1}.
\end{align*}
Systems with a Lax pair formulation are often described as completely integrable; indeed, it easily follows from \eqref{eqn:laxpair} that the operators $J(t)$ are mutually unitarily equivalent, which can be used to extract conserved quantities. This is particularly convenient for the periodic Toda lattice, which is a finite dimensional integrable system in which a maximal set of conserved quantities can be obtained as traces of powers of $J(t)$ \cite{vM76}. In the aperiodic case considered in this paper, one cannot rely on this method. Indeed, proving stronger statements of integrability is a problem highly dependent on the type of initial data considered.


In the context of the KdV equation, Deift \cite{DEI08} posed an open problem whether, for almost periodic initial data, solutions to the KdV equation were almost periodic in the time variable.  In \cite{BDGL15}, this question is partially answered in the affirmative; namely, that for small, quasi-periodic analytic initial data at Diophantine frequency, unique solutions exist and are almost periodic in time.  In this paper, we demonstrate that the methods of \cite{BDGL15} can be applied to the analogous question for the Toda lattice: for almost periodic initial data $(\tilde{a},\tilde{b})$, is it true that the solution to \eqref{eqn:todalattice1}, \eqref{eqn:todalattice2}, \eqref{eqn:todalatticeIV} is almost periodic in the time variable $t$? While the answer to this question was known (in a much stronger sense) to Vinnikov and Yuditskii \cite{VY02}, our main theorem provides sufficient spectral conditions on the initial data to guarantee this time almost periodicity indeed occurs.

We now describe an aspect of these spectral conditions. Denote by $J_0 = J(0)$ the Jacobi operator corresponding to the initial data $\tilde a, \tilde b$. The spectrum $E = \sigma(J_0)$ is compact, and can thus be written
\begin{align*}
E = [\underline{E},\overline{E}] \setminus \bigcup_{j \in I}(E^-_j,E^+_j)
\end{align*}
for an appropriate (at most countable) indexing set $I$.  Here, $\underline{E} = \inf E$, $\overline{E} = \sup E$, and $(E^-_j,E^+_j)$ are the bounded maximal open intervals in $\bbR \setminus E$, called gaps.  Denote the gap lengths by
\begin{align*}
\gamma_j = E_j^+ - E_j^-
\end{align*}
for $j \in I$, and the distances between gaps by
\begin{align*}
\eta_{j,l} &= \min\{|E_j^+ - E_l^-|,|E_j^- - E_l^+|\}, \\
\eta_j &= \min\{|E_j^+ - \overline{E}|,|E_j^- - \underline{E}|\}.
\end{align*}
It will be important for our spectrum to satisfy conditions analogous to that of Craig \cite{CRA89}.  If we denote
\begin{align}
\label{eqn:Cj}
C_j = ((\overline{E}-\underline{E})-\eta_j)^\frac{1}{2} \exp\left(\frac{1}{2} \sum_{k\neq j} \frac{\gamma_k}{\eta_{j,k}}\right).
\end{align}
we wish for the following spectral assumptions to hold:
\begin{align}
\label{eqn:craigcondn}
\sup_{j \in I} \gamma_jC_j < \infty, \hspace{1cm}
\sup_{j \in I} \frac{\gamma_j}{\eta_j}C_j < \infty, \hspace{1cm}
\sup_{j \in I} \sum_{k \neq j} \frac{(\gamma_j\gamma_k)^\frac{1}{2}}{\eta_{j,k}}C_j < \infty.
\end{align}
Heuristically, this condition guarantees that relatively large gaps of the spectrum do not accumulate.

In this context, we have a sufficient condition which demonstrates the time-and-space almost periodicity of the associated solution:

\begin{thm}
\label{thm:mainthm}
Let the initial data $J_0 = (\tilde{a},\tilde{b})$ be almost periodic.  Denote $E = \sigma(J_0)$, and assume that $E = \sigma_{ac}(J_0)$ and $E$ satisfies \eqref{eqn:craigcondn}. Then:
\begin{enumerate}
\item the unique solution $(a,b)(t)$ of the initial value problem \eqref{eqn:todalattice1},\eqref{eqn:todalattice2},\eqref{eqn:todalatticeIV} is bounded in time;
\item the solution is almost periodic in $t$, in the sense that there is a continuous map
\begin{align*}
\cM : \bbT^I \to \ell^\infty(\bbZ) \times \ell^\infty(\bbZ),
\end{align*}
a point $\omega \in \bbT^I$, and a direction $\zeta \in \bbR^I$ such that $(a,b)(t) = \cM(\omega + \zeta t)$;
\item for each $t \in \bbR$, the associated Jacobi operator $J(t)$ is almost periodic with frequency module equal to the frequency module of $J_0$.
\end{enumerate}
\end{thm}

\begin{rem}
After proving this theorem, we discovered that in fact a much stronger result was known to Vinnikov and Yuditskii \cite{VY02}: under milder spectral restrictions on the initial condition $J_0$, one has time and space almost-periodicity of solutions to any Lax pair \eqref{eqn:laxpair} where $P(t)$ is the skew-adjoint part of any bounded function $f \in L^\infty(E)$ applied to $J(t)$.  As such, our Theorem \ref{thm:mainthm} should be viewed as a discrete demonstration of the techniques developed in \cite{BDGL15}.

In the same breath, our approach also produces some auxiliary results which may be of independent interest and are new in the literature.  In particular, Proposition \ref{prop:disczeros} establishes the non-pausing of Dirichlet eigenvalues at gap edges under the Toda flow by proving the existence of subordinate gap-edge solutions under the Craig-type conditions above; consequently, we find in Proposition \ref{prop:dirichtimederiv} that the Dirichlet data in this setting flow in analog to the usual Dubrovin formulas.
\end{rem}

The proof of this theorem relies heavily on the previous work establishing analogous results for quasi-periodic initial data with finite-gap spectra; these results are collected in, e.g., \cite{GHMT08}.  It also relies on the inverse spectral-theoretic works of Remling \cite{REM11,REM15} and Sodin-Yuditskii \cite{SY95,SY97}.  These results and their relevance will be discussed in further detail later in the paper.

On comparison with the continuum analogue \cite{BDGL15}, one may notice that the application to small quasiperiodic initial data \cite[Theorem 1]{BDGL15} is conspicuously absent.  The proof of this application was due in large part to extensive machinery developed in \cite{DG14, DG15, DGL14.2, DGL14.3, DGL14.1} for the continuum Schr\"odinger operator.  The importance of this machinery is that it provides a direct spectral criterion guaranteeing that the analogue of \eqref{eqn:craigcondn} is satisfied.  With this in mind, we remark that the analogue of \cite[Theorem 1]{BDGL15} is very likely true in the Toda lattice case, pending a number of discrete analogues to continuum results.  One would then apply comparable results to \cite{DG14, DG15, DGL14.2, DGL14.3, DGL14.1} to small quasi-periodic initial data $J_0$ to verify this assumption, consequently concluding time almost periodicity.  Proving these discrete analogues, however, seems an undertaking well beyond the scope of this paper.

The paper will be structured as follows: in Section 2, we establish some definitions and notation and review some fundamental results in inverse spectral theory.  In Section 3, we review the cocycle representation of the Toda flow and describe the time evolution of the Weyl $M$-matrix.  Section 4 addresses perhaps the primary difficulty by demonstrating that the Dirichlet data corresponding to the Toda flow do not pause at the edges of gaps.  In Section 5, we prove a relation describing the Toda flow on the Dirichlet data; in particular, we prove that the Dirichlet data flow in a differentiable manner with respect to a Lipschitz vector field.  Finally, Section 6 addresses the linearization of the Toda flow under the Abel map, and contains the proof of Theorem \ref{thm:mainthm}.  For the curious reader, we have also appended a proof that reflectionless Jacobi operators have spectrally-everywhere vanishing Lyapunov exponents in the Sodin-Yuditskii regime (Appendix A).

\section{Background and Notation}

We consider bounded Jacobi operators, which are operators $J = J(a,b)$ on $\ell^2(\bbZ)$ parametrized by a pair of bounded, real-valued sequences $a,b \in \ell^\infty(\bbZ)$ as follows:
\begin{align}
\label{eqn:jacobiop}
(Ju)_n = a_nu_{n+1}+b_nu_n+a_{n-1}u_{n-1}.
\end{align}
Jacobi operators arise naturally in the context of the spectral theorem: any bounded self-adjoint operator $A$ with a cyclic vector is unitarily equivalent to a Jacobi operator on a half-line. They also correspond closely to orthogonal polynomials on the real line and generalize discrete one-dimensional Schr\"odinger operators (for which $a_n = 1$ for all $n \in \bbZ$).
We restrict to the non-singular case, where $a_n > 0$ for all $n \in \bbZ$.

It is easy to see that, under these conditions, $J$ is a bounded self-adjoint operator.  The spectrum $\sigma(J)$ is the set of $z \in \bbC$ for which $J-z$ does \textit{not} have a bounded inverse operator $(J-z)^{-1}: \ell^2(\bbZ) \to \ell^2(\bbZ)$.  Because $J$ is self-adjoint, $\sigma(J) \subset \bbR$, and because $J$ is bounded, $\sigma(J)$ is compact.

Associated to each Jacobi operator $J$ is its resolvent, $(J-z)^{-1}$, which (by definition) is a bounded operator for $z \in \mathbb{C} \setminus \sigma(J)$, the resolvent set.  This operator can likewise be put in a matrix form, with elements given by
\begin{align}
\label{eqn:greenfn}
r(n,m;J,z) = \langle \delta_n, (J-z)^{-1} \delta_m \rangle
\end{align}
We write $r(n;J,z) := r(n,n;J,z)$ for the diagonal elements of the resolvent matrix, and $r(J,z) := r(0;J,z)$ for the spectral theoretic Green's function of $J$.  This function is holomorphic on the resolvent set $\rho(J):= \mathbb{C}\setminus \sigma(J)$.

For $\psi \in \ell^2(\bbZ)$, the spectral measure $d\mu_\psi$ is the unique measure on $\bbR$ with the property that
\begin{align}
\label{eqn:spectmeas}
\langle \psi, (J-z)^{-1}\psi \rangle = \int \frac{1}{x-z} \, d\mu_\psi(x), \; \forall z \in \bbC \setminus \bbR.
\end{align}
Spectral measures are always supported on the spectrum.  Using the Lebesgue decomposition of $d\mu_\psi = d\mu_{\psi,ac} + d\mu_{\psi,s}$, the absolutely continuous spectrum $\sigma_{ac}(J)$ can be defined as the smallest common topological support of absolutely continuous parts of all spectral measures,
\begin{align*}
\sigma_{ac}(J) = \overline{\bigcup_{\psi \in \ell^2(\bbZ)} \supp d\mu_{\psi,ac}}
\end{align*}
Clearly, $\sigma_{ac}(J) \subset \sigma(J)$.  We will focus on cases where $\sigma_{ac}(J) = \sigma(J)$.

For each $z \notin \sigma(J)$, the difference equation $Ju = zu$ has nontrivial (formal) solutions $u_\pm(J;z)$, called Weyl solutions, such that $u_\pm(J;z) \in \ell^2(\bbZ_\pm)$. The Weyl solutions are clearly linearly independent, and thus span the 2-dimensional solution space $\ker(J-z)$.  Furthermore, they are unique up to a multiplicative constant.  For simplicity of notation, we write
\begin{align}
\label{eq:vplusminus}
V_\pm(J;z) &= \begin{bmatrix}
u_\pm(J;z)(1) \\
a_0(J)u_\pm(J;z)(0)
\end{bmatrix}.
\end{align}
The resolvent function then satisfies the important relation
\begin{align*}
r(n,m;J,z) = \dfrac{u_\pm(n)u_\mp(m)}{W(u_-,u_+)},
\end{align*}
where $\pm 1 = \sgn(n-m)$ and $W(u_-,u_+) = a_0(u_-(0)u_+(1) - u_-(1)u_+(0))$ denotes the Jacobi Wronskian.

In terms of the Weyl solutions, one can express the half-line $m$-functions by
\begin{align*}
m_\pm(J;z) = \mp \frac{u_\pm(J;z)(1)}{a_0(J)u_\pm(J;z)(0)}.
\end{align*}
These functions are the half-line analogues of the spectral Green's function.  For fixed $J$, these are meromorphic functions of $z \in \bbC \setminus E$ which analytically map the upper half-plane to itself.  We write
\begin{align*}
M_\pm(J;z) = \begin{bmatrix}
m_\pm(J;z) \\
1
\end{bmatrix},
\end{align*}
and remark that $V_\pm(J;z) = \mp M_\pm(J;z)$ in $\bbC\bbP^1$.  Note the important equalities
\begin{align}
\label{eqn:greenstomfns}
r(0;J,z) &= \frac{1}{a_0^2(J)(m_+(J,z) + m_-(J,z))}, \\
r(1;J,z) &= \frac{m_-(J,z)m_+(J,z)}{m_+(J,z)+m_-(J,z)}.
\end{align}
It can be shown that $r(J,z)$ has at most one simple zero  $\mu_j \in (E_j^-,E_j^+)$ for each $j \in I$.  When such a zero exists, it is likewise a pole of exactly one of $m_\pm(J;z)$.  Write $\sigma_j = \pm1$ in this case.  More precisely, fixing $J$ and observing that $r(n,n;z)$ is strictly increasing for $z \in (E_j^-,E_j^+)$, we define
\begin{align*}
\mu_j(n) = \begin{cases}
z \in (E_j^-,E_j^+) & r(n,n;z) = 0 \\
E_j^+ & r(n,n;z) < 0  \;\; \forall z \in (E_j^-,E_j^+) \\
E_j^- & r(n,n;z) > 0 \;\; \forall z \in (E_j^-,E_j^+)
\end{cases}
\end{align*}

In the first case, the sign $\sigma_j$ is defined so that $m_{\sigma_j}$ has a pole at $\mu_j$; in the latter two cases, we don't have this dichotomy, so we say $\sigma_j = 0$. The pairs $\{(\mu_j(n), \sigma_j(n))\}_{(j,n) \in I \times \bbZ}$ form the so-called Dirichlet data of $J$.

Define the isospectral torus $\cD(E) = \bbT^I$ with the metric
\begin{align}
\label{eqn:metric}
\|\varphi - \tilde{\varphi}\|_{\cD(E)} = \sup_{j\in I} \gamma_j^\frac{1}{2}\|\varphi_j - \tilde{\varphi}_j\|_{\bbT}.
\end{align}
Here, we denote by $\bbT$ the topological circle $\bbR/2\pi\bbZ$.  Introduce variables $\varphi$ on $\cD(E)$ given implicitly by:
\begin{align}
\label{eqn:mutophi}
\mu_j &= E_j^- + (E_j^+ - E_j^-)\cos^2(\frac{\varphi_j}{2}) \\
\label{eqn:sigma}
\sigma_j &= \begin{cases}
+1 & \varphi_j \in (0,\pi) + 2\pi\bbZ \\
-1 & \varphi_j \in (-\pi,0) + 2\pi\bbZ \\
0 & \varphi_j \in 2\pi\bbZ
\end{cases}.
\end{align}
Denote by $\varphi(n) := (\varphi_j(n))_{j \in I}$ the angular coordinates corresponding to $\mu_j(n)$.

The tangent space of a point on $\cD(E)$ will be equipped with the norm
\begin{align}
\label{eqn:tangentspacenorm}
\|v\| = \sup_{j \in I} \gamma_j^\frac{1}{2}|v_j|
\end{align}
and vector fields on $\cD(E)$ will be equipped with the sup-norm obtained from \eqref{eqn:tangentspacenorm}.  Denote the map taking a Jacobi operator $(a,b)$ to its Dirichlet data at the origin by $\cB$:
\begin{align}
\label{eqn:optodiv}
\cB((a,b)) := \varphi(0).
\end{align}

We introduce a vector field $\Psi$ via
\begin{align*}
\Psi_j(\varphi) &=  2\Big( (\underline{E} - \mu_j )(\overline{E} - \mu_j) \prod_{k \neq j} \frac{(E_k^- - \mu_k)(E_k^+ - \mu_k)}{(\mu_k - \mu_j)^2} \Big)^{\frac{1}{2}}.
\end{align*}
 We will soon see that this vector field describes the Toda flow in the discrete case. Craig initially proposed a similar vector field to describe the translation flow in the continuum case \cite{CRA89}.  Roughly speaking, $\Psi_j(\varphi)$ is the residue of $(r(J,z))^{-1}$ at $\mu_j$.  This is not quite true in general -- $r(J,z)$ does not have proper zeroes at the gap edges -- but at least motivates the definition.

To ensure that the vector field $\Psi$ is Lipschitz, we also enforce condition \eqref{eqn:craigcondn} on $E$.  This condition prevents relatively large gaps from accumulating at a gap edge; we show this is analogous to Craig's conditions \cite[Theorem 6.2]{CRA89} in Lemma \ref{lem:craigequiv}.

We are interested particularly in the class of Jacobi operators which can be completely recovered from their Dirichlet data.  To this end, we introduce the following

\begin{defn} Let $E \subset \mathbb{R}$. A Jacobi operator $J$ is said to be \textit{reflectionless on $E$} if for every $n \in \bbZ$
\begin{align}
\label{eqn:reflectionless}
\Re(r(n;J,x+i0))=0
\end{align}
for Lebesgue almost every $x \in E$.
\end{defn}
The importance of the reflectionless condition is that it allows one to recover a Jacobi operator's resolvent $r(n;z)$ from its spectral data $\{E_j^\pm\}\cup \{(\mu_j(n),\sigma_j(n))\}_{j\in I}$.


For a fixed positive-measure compact set $E \subset \mathbb{R}$, define
\begin{align*}
\cJ(E) :&= \{(a,b) \in \ell^\infty(\mathbb{Z}) \times \ell^\infty(\mathbb{Z}) : \sigma(J(a,b)) \subset E, \; J(a,b) \text{ is reflectionless on } E\}.
\end{align*}
 If $E$ has $|I| < \infty$ gaps, $\cJ(E)$ can be parametrized by the real divisors on the two-sheeted genus $|I|$ Riemann surface $\Sigma_{I}$ covering $\hat{\bbC}\setminus E$.  Thus, $\cJ(E)$ is homeomorphic to the real Jacobian of $\Sigma_{I}$, which is in turn homeomorphic (via the so-called Abel map) to a torus in $|I|$ dimensions, $\bbT^{I}$, by the Abel-Jacobi theorem \cite{MIR95}.

In a remarkable paper \cite{SY97}, Sodin and Yuditskii showed that, if $E$ is homogeneous in the sense of Carleson, a generalized Abel map $\cA$ gives a homeomorphism between $\pi^*(\bbC \setminus E) := \Hom(\pi_1(\bbC \setminus E),\bbT)$ and $\cD(E)$, and the restriction $\cB : \cJ(E) \to \cD(E)$ of \eqref{eqn:optodiv} is likewise a homeomorphism.  What's more, they provide continuous trace formulas $\cP, \cQ : \cD(E) \to \bbR$ allowing recovery of the parametrizing sequences (see Appendix A).  We denote
\begin{align}
\label{eqn:trform}
\cV(\varphi) := (\cP(\varphi),\cQ(\varphi)).
\end{align}

When $E$ is homogeneous (in the sense of Carleson), the shift action on $\cJ(E)$ is completely understood: namely, shifting corresponds exactly to translation in $\pi^*(\bbC \setminus E)$ by the constant vector $\alpha_E$.  Our condition \eqref{eqn:craigcondn} is, in fact, stronger than homogeneity (see Appendix in \cite{BDGL15}); thus, for $E$ satisfying \eqref{eqn:craigcondn}, we are indeed in Sodin-Yuditskii's regime.

The idea behind the proof of Theorem \ref{thm:mainthm} can be summarized as follows: Our assumptions allow us to conclude that our initial data lies in $\cJ(E)$, that $E$ is homogeneous, and that the trace formulas $\cP$ and $\cQ$ are valid.  This will suffice to prove (1).  We use this information to pass from our initial data to the Dirichlet data, and conclude global existence and uniqueness of the flow there.  To show (2), we will approximate our infinite-gap operators in Dirichlet data via lifts of finite-gap operators, where this theorem is known, into our larger isospectral torus, and use uniform convergence to conclude that the result must still hold.  Finally, (3) comes for free as an immediate corollary of Sodin-Yuditskii \cite[Corollary of Theorem C]{SY97}.  The remainder of this paper is, of course, dedicated to the details.

\section{Evolution of the Weyl M-matrix Under the Toda Flow}

In this section, we review a number of prior results, which can be found, for example, in \cite{TES00}.

Under the assumptions above, we have that $\cJ(E)$ is homeomorphic to a torus of some dimension.  A fundamental action on $\cJ(E)$ is given by conjugation by the shift operator; namely, if $J = J(a,b)$, we are interested in $SJS^* = J(Sa,Sb)$, where $S:\delta_n \mapsto \delta_{n-1}$.  We abuse notation and write $SJ$ for this action.  By fixing a solution space and canonical basis, it is possible to explicitly express the action of $S$ on $\ker(J-z)$. Namely, for $z \in \bbC$, in the 2-dimensional solution space $\ker(J-z) \subset \bbC^\bbZ$ we have a correspondence of any solution to a vector in $\bbC^2$, given by
\begin{align*}
u \leftrightarrow \begin{bmatrix}
u(1) \\
a_0u(0)
\end{bmatrix}.
\end{align*}
The corresponding basis of solutions is given by solutions $e_1$ and $e_2$ such that
\begin{align}
\label{eqn:e1}
e_1(J;z) &\leftrightarrow \begin{bmatrix}
1 \\
0
\end{bmatrix} \\
\label{eqn:e2}
e_2(J;z) &\leftrightarrow \begin{bmatrix}
0 \\
1
\end{bmatrix}
\end{align}
and the corresponding representation of $S : \ker(J-z) \to \ker(SJ - z)$ is
\begin{align*}
S(J;z) &:= \begin{bmatrix}
 Se_1(J;z) \; \vert \; Se_2(J;z)
 \end{bmatrix} \\
 &= \dfrac{1}{a_1(J)}\begin{bmatrix}
z-b_1(J) & -1 \\
a^2_1(J) & 0
\end{bmatrix}.
\end{align*}
$S(J;z)$ is commonly referred to as a 1-step transfer matrix.  

We intend to study a differentiable $\bbR$-action on $\cJ(E)$ induced by the Toda lattice \eqref{eqn:todalattice1}, \eqref{eqn:todalattice2}.  This action, initially characterized via equations relating the positions and momenta of particles on a line, was shown to be equivalent to a Lax pair \eqref{eqn:laxpair}, where $J(t) = (a,b)(t)$ and $P(t)$ is the skew symmetric operator given by
\begin{align*}
(P(t)u)_n = -a_{n-1}(t)u_{n-1} + a_n(t)u_{n+1}.
\end{align*}
As a consequence of the Lax pair formalism, for any solution $J(t)$ to \eqref{eqn:laxpair}, for each $t_0, s_0 \in \bbR$, $J(t_0)$ is unitarily equivalent to $J(s_0)$ via a unitary propagator.  We are thus interested in differentiable solutions $u(t)$ to $J(t)u(t) = \lambda u(t)$; in particular, those satisfying
\begin{align}
\label{eqn:solnderiv}
\partial_t u(t) = P(t)u(t)
\end{align}
will be of some importance.  Such solutions are uniquely determined by an initial value \cite[Lemma 12.15]{TES00}.  

The key feature of the Toda lattice flow that we will use is the following:
\begin{prop}\cite[Proposition 1.3, Proposition 1.4]{REM15}
\label{prop:toda}
Denote by $J(t)$ the unique solution to \eqref{eqn:todalattice1},\eqref{eqn:todalattice2}, with initial condition $J \in \cJ(E)$.  Then $J(t) \in \cJ(E)$ for all $t \in \bbR$.
\end{prop}

%
%

We now prove some relevant facts regarding the evolution of this system.  First, we will require the following
lemma showing time-invariance of subordinacy (in the sense of Gilbert and Pearson \cite{GP87}); in particular, this means that if $u$ satisfies \eqref{eqn:solnderiv} and $u(0) = u_+$, then $u(t) = u_+(t)$ is the Weyl solution for all time \cite[Lemma 12.16]{TES00}:

\begin{lem}
\label{lem:subordinv}
Suppose $(a,b) :\bbR \to \ell^\infty(\bbZ) \times \ell^\infty(\bbZ)$ is
the solution
 of \eqref{eqn:todalattice1}, \eqref{eqn:todalattice2}, and $z \in \bbC$.  Then subordinacy at $\pm \infty$ is time-invariant; i.e., if $u_0$ solving $J_0u = z u$ is subordinate at $\pm \infty$, then $u(t)$ solving $J((a,b)(t))u(t) = z u(t)$, $\partial_t u(t) = P(t)u(t)$, $u(0) = u_0$ is subordinate at $\pm\infty$ for all $t \in \bbR$.
\end{lem}

\begin{proof}
Denote by $J(t) := J((a,b)(t))$ the Jacobi operator associated to $(a,b)(t)$, fix $z \in \bbC$, and let $K(t) = \ker(J(t)-z)$.  We treat only the subordinacy at $+\infty$ case, noting the $-\infty$ case is exactly analogous.  Define a norm $\|\cdot\|_L$ on $\bigcup_{t \in \bbR} K(t)$ by
\begin{align*}
\|u\|_L^2 = \sum_{j=0}^{\lfloor L \rfloor} |u_j|^2 + (L-\lfloor L \rfloor)|u_{\lfloor L \rfloor + 1}|^2,
\end{align*}
where $\lfloor \cdot \rfloor$ denotes the typical floor function.  Suppose also that $u_0 \in K := K(0)$ is subordinate, i.e. for any $v_0 \in K$ which is linearly independent of $u_0$,
\begin{align}
\label{eqn:subordinatedef}
\lim_{L \to \infty} \frac{\|u_0\|_L}{\|v_0\|_L} = 0.
\end{align}

Since $P$ is uniformly bounded operator and $u$ is a solution of $Ju = z u$, there is some constant $C(t) > 0$ such that $|(Pu)_j|\leq C(t)(|u_j|+|u_{j+1}|)$ for all $u \in K(t)$, for all $j \in \bbZ$.  In fact, since our parametrizing sequences $a(t), b(t)$ are uniformly bounded in $t$ and $P$ depends polynomially on these sequences, this constant $C(t)$ is in fact uniformly bounded (say by $C$) in $t$.  Thus, if $u(t) \in K(t)$ is a solution, if we write $S(L,t) := \|u(t)\|^2_L$ (with $L \in \bbZ$ for simplicity), then we have
\begin{align}
\label{eqn:Sderivbd}
|2\sum_{j=0}^L\Re(\overline{u_j(s)}(Pu)_j(s))| &\leq 8C(S(L,t) + |u_{L+1}(s)|^2) \\
&\leq \tilde{C}S(L,t)
\end{align}
where the first inequality in the comes from the Cauchy-Schwartz inequality and the definition of $P$ \cite[Equation 12.96]{TES00}, and the second comes from the recursion relation $Ju = \lambda u$ and the uniform boundedness of the parametrizing sequences $(a,b)$.

Thus, by \eqref{eqn:Sderivbd}, we have
\begin{align*}
|\frac{d}{dt}S(L,t)| &= |2\sum_{j=0}^L\Re(\overline{u_j(s)}(Pu)_j(s))| \\
&\leq \tilde{C}S(L,t)
\end{align*}
and, by Gronwall's inequality, it follows that
\begin{align}
\label{eqn:Lnormbd}
S(L,0)\exp(-\tilde{C}t)\leq S(L,t) \leq S(L,0)\exp(\tilde{C}t).
\end{align}

Suppose $u(\cdot) : \bbR \to K(\cdot)$ solves $\partial_t u =Pu$, and that $u_0 = u(0)$ is subordinate at $+\infty$.  Let $v_{t_0} \in K(t_0)$ be linearly independent of $u(t_0)$.  There exists a unique $v(t)$ solving $\partial_t v = Pv$, $J(t)v(t) = z v(t)$ such that $v(t_0) = v_{t_0}$; what's more, this solution is linearly independent of $u(t)$ for all time by uniqueness of weak solutions \cite[Theorem 12.5]{TES00}.  In particular, $v_0 := v(0)$ is linearly independent of $u_0$.  Then, by \eqref{eqn:Lnormbd} and subordinacy of $u_0$,

\begin{align*}
\lim_{L\to \infty}\frac{\|u(t)\|_L}{\|v(t)\|_L} &\leq \lim_{L \to \infty} \frac{\exp(2\tilde{C}t)\|u_0\|_L}{\|v_0\|_L} \\
 &= 0
\end{align*}
i.e., $u(t)$ is subordinate at $+\infty$ in $K(t)$.
\end{proof}

As a consequence of the time-invariance of Weyl solutions, we find the following

\begin{prop}
\label{pr:Amat}
There is a function $A: \bbR \times \cJ(E) \times (\bbC \setminus E) \to \mathfrak{sl}_2(\bbC)$ given by
\begin{align}
\label{eqn:derivfactmat}
A(t,J;z) = \begin{bmatrix}
(z-b_1(t)) & -2 \\
2a^2_0(t) & -(z-b_1(t))
\end{bmatrix},
\end{align}
such that
\begin{align}
\label{eqn:cocyclederiv}
\partial_t V_\pm(t) = A(t)V_\pm(t).
\end{align}
Here, $V_\pm$ are vectors \eqref{eq:vplusminus} corresponding to Weyl solutions of $J(t)u = zu$ satisfying \eqref{eqn:solnderiv}, where $J(t) = J(a,b)(t)$ is the unique solution to \eqref{eqn:todalattice1},\eqref{eqn:todalattice2}, with initial condition $J \in \cJ(E)$.
\end{prop}

\begin{proof}
Fix $z \in \bbC \setminus E$.  A straightforward calculation shows that $P(t;z) := P(t)|_{\ker(J(t)-z)}$ satisfies
\begin{align}
\label{eqn:laxoptes}
P(t;z) = 2a(t)S - (z-b(t)).
\end{align}
Differentiating $V_\pm(t)$, the entries $A_{ij}$ follow from simple computations using \eqref{eqn:todalattice1}, \eqref{eqn:laxoptes}, and $u_\pm \in \ker(J(t)-z)$.
\end{proof}

This leads to the following easy

\begin{prop}\cite[Lemma 12.15]{TES00}
\label{pr:wronskind}
For Weyl solutions $u_\pm$ satisfying \eqref{eqn:solnderiv}, the Wronskian $W(u_-,u_+)$ is independent of time:
\begin{align*}
\partial_t W(u_-,u_+) = 0.
\end{align*}
\end{prop}

\begin{proof}
Writing $V = \begin{bmatrix}
V_+ \; | \; V_-
\end{bmatrix}$, it is easily seen that
\begin{align}
W(u_-,u_+) = \det\left( V \right).
\end{align}
Taking derivatives, we see that
\begin{align}
\partial_t W(u_-,u_+) &= \det(V)\tr(V^{-1}AV) \\
&= \det(V)\tr(A) \\
&= 0.
\end{align}
Here, the first equality is Jacobi's formula, the second is by invariance of the trace under cyclic permutations, and the third is by $A \in \mathfrak{sl}_2(\bbC)$.
\end{proof}

Define the Weyl $M$-matrix of a Jacobi operator as
\begin{align*}
M(J,z) &= \begin{bmatrix}
r(1,1;J,z) & r(1,0;J,z) \\
r(0,1;J,z) & r(0,0;J,z)
\end{bmatrix}.
\end{align*}
We have the following theorem describing the time-evolution of $M$:

\begin{thm}
\label{thm:Mmatrixderiv}
Denote by $J(t)$ the Jacobi matrix corresponding to the unique solution $(a,b)(t)$ of \eqref{eqn:todalattice1}, \eqref{eqn:todalattice2} with initial condition $J$.  Then for each $z \in \bbC \setminus \sigma(J)$ the Weyl $M$-matrix of $J(t)$ is differentiable in time, with derivative
\begin{align}
\label{eqn:Mmatrixtderiv}
\partial_tM = BM+MB^\top,
\end{align}
where $B$ is given by
\begin{align*}
B(t;z) &= \begin{bmatrix}
(z-b_1(t)) & -2a_0(t) \\
2a_0(t) & -(z-b_0(t))
\end{bmatrix}.
\end{align*}
\end{thm}

\begin{proof}
We fix $z \in \bbC$ and suppress its notation.  Write
\begin{align*}
A_0(t) &= \begin{bmatrix}
1 & 0 \\
0 & a_0^{-1}(t)
\end{bmatrix}.
\end{align*}
Since the Weyl $m$-functions are independent of the normalization of the Weyl solutions, we may choose those solutions $u_\pm(t)$ satisfying \eqref{eqn:solnderiv}; with this choice, the differentiability of $M$ is clear, and one can check that
\begin{align*}
M(t) :&= M(J(t)) \\
&= \frac{1}{2W(u_-,u_+)}A_0(t)\left(V_+(t)V_-^\top (t) + V_-(t)V_+^\top (t) \right)A_0^\top(t).
\end{align*}
Differentiating and using Proposition \ref{pr:wronskind}, we see that
\begin{align*}
\partial_t M &= BM + MB^\top 
\end{align*}
with $B$ given by
\begin{align*}
B(t) &= A_0(t)A(t) A_0^{-1}(t) + \partial_t A_0(t) A_0(t)^{-1}.
\end{align*}
Using \eqref{eqn:todalattice1} and Proposition \ref{pr:Amat} yields the claimed form for $B$.
\end{proof}

\begin{rem}
As noted at the beginning of this section, many of these results can be found in some form in, e.g., \cite[Chapter 12]{TES00}.  However, we have not seen Theorem \ref{thm:Mmatrixderiv} presented quite in this form.  The advantage of presenting the time evolution in the form of Theorem \ref{thm:Mmatrixderiv} is that this theorem is the analogue of \cite[Theorem 1]{RYB08}.
\end{rem}

\section{The Non-Pausing of Dirichlet Data in the Non-Stationary Toda Flow}

We wish to prove a Dubrovin-type evolution formula for the angular coordinates describing the Dirichlet data of the Jacobi operators under the Toda flow \eqref{eqn:todalattice1}, \eqref{eqn:todalattice2}.  To this end, it will become important to show that the Dirichlet data do not pause at the gap edges.  As long as $(a,b)$ is a non-constant (i.e., non-stationary) solution, this pausing indeed does not occur.  We clarify these statements here.

The following Proposition establishes exactly the non-pausing of the Dirichlet data at gap edges under the Toda flow:

\begin{prop}
\label{prop:disczeros}
Assume that $E$ satisfies \eqref{eqn:craigcondn} and $(a,b): \bbR \to \cJ(E)$ is a non-constant solution satisfying \eqref{eqn:todalattice1},\eqref{eqn:todalattice2}.  Fix a $j \in I$ and let $\lambda \in \{E_j^\pm\}$.  Then the set $\{t \in \bbR : \mu_j(t) = \lambda\}$ is discrete.
\end{prop}

To prove this proposition, we require the following lemma establishing the real analyticity in time of solutions:

\begin{lem}
\label{lem:analytic}
Suppose $E$ satisfies \eqref{eqn:craigcondn}.  Let $J(\cdot) : \bbR \to \cJ(E)$ correspond to a solution to \eqref{eqn:todalattice1},\eqref{eqn:todalattice2}, $\lambda \in E$, and let $u(t)$ solve $J(t)u = \lambda u$, $\partial_t u(t) = P(t)u(t)$.  Then, for fixed $n \in \bbZ$, the terms $u_n(t)$ are real analytic in $t$.
\end{lem}

\begin{proof}
Because $u(t)$ solves $J(t)u = \lambda u$ and $\partial_t u = Pu$, it is differentiable in $t$, with derivative
\begin{align}
\label{eqn:subdtsolnderiv}
\frac{d}{dt}u_n(t) = 2a_n(t)u_{n+1}(t) - (\lambda - b_n(t))u_n(t).
\end{align}
One recursively finds that $u_n(t) \in C^\infty(\bbR,\bbR)$.  Now, to check analyticity, it suffices to show that
\begin{align}
\label{eqn:uanalytic}
\limsup_{m \to \infty} \left( \frac{\sup_{t \in [-t_0, t_0]}|\frac{d^m}{dt^m}u_n(t)|}{m!} \right)^\frac{1}{m} < \infty.
\end{align}
for each $t_0 > 0$.

First, note that because $J(t) \in \cJ(E)$ for all $t \in \bbR$ and $E$ is compact, the parametrizing sequences $(a,b)$ are uniformly bounded by some constant $C$.  Furthermore, because $(a,b)$ solve \eqref{eqn:todalattice1},\eqref{eqn:todalattice2}, they are smooth and have derivatives which are polynomials in $a$ and $b$ of total degree 2.  At each step of differentiation, the number of terms in the polynomial doubles, and the degree of each new term increases by 1.  Thus, letting $D = 4C$, one can then check via a straightforward induction that
\begin{align}
\label{eqn:aderivbd}
\|\frac{d^m}{dt^m}a(t)\| &\leq m!D^{m+1} \\
\label{eqn:bderivbd}
\|\frac{d^m}{dt^m}b(t)\| &\leq m!D^{m+1}
\end{align}

Because $E$ satisfies \eqref{eqn:craigcondn}, $E$ is homogeneous due to an observation of Sodin \cite{BDGL15,SodinUnp}.  It follows that the Lyapunov exponent agrees with the value of the spectral Green's function on $E$ (by Proposition \ref{prop:zerole}).  In particular, we have that the Lyapunov exponent at spectral energies must be zero, i.e. the solution $u$ cannot grow exponentially in the spatial variable $n$.

Equations \eqref{eqn:aderivbd} and \eqref{eqn:bderivbd} and the sub-exponential growth of the terms $u_n$ combined with the differential relation \eqref{eqn:subdtsolnderiv} imply that \eqref{eqn:uanalytic} holds, and thus $u_n(t)$ is indeed real analytic.
\end{proof}

With Lemmas \ref{lem:subordinv} and \ref{lem:analytic} in hand, we can address the

\begin{proof}[Proof of Proposition \ref{prop:disczeros}]
Separating the term $l = j$ in the product formula \eqref{eqn:greenfn}, we write
\begin{align*}
r(n,n;z,t) = \frac{1}{2} \sqrt{\frac{(\mu_j(n,t) - z)^2}{(E_j^+ - z)(E_j^- -z)}}\sqrt{\frac{1}{(\underline{E}-z)(\overline{E}-z)}\prod_{l \neq j} \frac{(\mu_l(n,t) - z)^2}{(E_l^- - z)(E_l^+ -z)}}.
\end{align*}
We can do this because $(a,b)$ is non-constant.  Note that, as $z \to \lambda$, if $\Re(z) \in [E_j^-,E_j^+]$, the product over $l \neq j$ has a finite limit (since the $l^{th}$ term is bounded in modulus by $1 + \gamma_l/\eta_{l,j}$).  We can therefore conclude that
\begin{align}
\label{eqn:gapedgegreens}
|r(n,n;\lambda+i0,t)| = \begin{cases}
0 & \mu_j(n,t) = \lambda \\
+\infty & \mu_j(n,t) \neq \lambda
\end{cases}.
\end{align}
Thus, if $\mu_j(n_0,t_0) = \lambda$ for some $(n_0, t_0) \in \bbZ \times \bbR$, we have that one of $m_\pm(\lambda+i0) = \infty$ by the relationship \eqref{eqn:greenstomfns}.  By shifting and flowing appropriately, suppose without loss of generality that $\mu_j(0,0) = \lambda$.  Suppose first that $m_+(\lambda+i0) = \infty$.  Then, by the important inequality of Jitomirskaya-Last \cite[Theorem 1.1]{JL99}, \cite{JL00}, it follows that, up to a multiplicative constant, the solution $e_1(J(0); \lambda)$ is the unique nontrivial solution subordinate at $+\infty$ for the energy $\lambda$.  Ostensibly, there may be some $t_0$ such that $\mu_j(0,t_0) = \lambda$ and $m_-(\lambda + i0) = \infty$.  Again by Jitomirskaya-Last, it follows that the solution $e_1(J(t_0);\lambda)$ is the unique nontrivial solution subordinate at $-\infty$  for $\lambda$.  In either case, denote by $\tilde{u}_+(t)$ the unique weak solution $\tilde{u}_+$ of $J(t)u = \lambda u$, $\partial_t u = P(t)u$ with $\tilde{u}_+(0) = e_1(J(0); \lambda)$, and by $\tilde{u}_-(t)$ the unique weak solution with $\tilde{u}_-(t_0) = e_1(J(t_0); \lambda)$.

Noting that subordinacy is invariant under the shift and Toda flow by Lemma \ref{lem:subordinv}, then by the uniqueness of the subordinate solution and Jitomirskaya-Last, it follows that
\begin{align}
\label{eqn:disczeros}
\{t \in \bbR : \mu_j(t) = \lambda \} \subset \{t \in \bbR : \tilde{u}_{+,0}(t) = 0\} \cup  \{t \in \bbR : \tilde{u}_{-,0}(t) = 0\},
\end{align}
where one may have to exclude one of the two sets on the right-hand side in the event the corresponding solution $\tilde{u}_\pm$ does not exist.

We now prove this set is discrete.  By our assumptions, Lemma \ref{lem:analytic} applies to $\tilde{u}_\pm$, and the zeros of $\tilde{u}_\pm$ are either discrete or $\tilde{u}_{\pm,0}(t) = 0$ for all $t$.  Suppose we are in the latter case.  Then $\partial_t \tilde{u}_{\pm,0}(t) = 0$ for all $t$.  But then, because $\tilde{u}_{\pm,0}(t) = 0$ for all $t$,
\begin{align*}
\partial_t \tilde{u}_{\pm,0}(t) &= 2a_n(t)\tilde{u}_{\pm,1}(t) \\
&= 0
\end{align*}
However, $\tilde{u}_\pm$ is non-trivial, so $\tilde{u}_{\pm,1}(t) \neq 0$, and $a_n(t)$ is non-zero by the non-singularity of $(a,b) \in \cJ(E)$.  This is a contradiction.  Thus, it follows that the set $\{t \in \bbR : \mu_j(t) = \lambda \}$ is subset of a union of two discrete sets, and thus is itself discrete.
\end{proof}

\begin{rem}
The thorough reader may notice that the proof of Proposition \ref{prop:disczeros} varies significantly from that of its continuum analogue, \cite[Proposition 2.1]{BDGL15}.  There, a critical step in the proof relied on oscillation theory to find a site where the Dirichlet eigenvalue $\mu_j$ lies within the open gap $(E_j^-,E_j^+)$.  We avoid this issue completely via the non-triviality of the subordinate solution and the differential relation $\partial_t u = Pu$.  However, our proof fails to show something we suspect to be true: that the ``pair" $u_\pm$ is in fact just one solution, subordinate at both $\pm \infty$.
\end{rem}

\section{A Dubrovin-Type Formula for the Toda Flow}

We wish to implement a vector field over $\cD(E)$ which describes the time evolution of the Dirichlet data via the Toda flow.  We begin by proving an analogue of \cite[Theorem 6.2]{CRA89}:

\begin{lem}
\label{lem:craigequiv}
If $E$ satisfies \eqref{eqn:craigcondn}, then $\Psi$ is a Lipschitz vector field on $\cD(E)$.
\end{lem}

\begin{proof}
This proof is simply the analogue of \cite[Theorem 6.2]{CRA89}.  Recall our terms $C_j$ from \eqref{eqn:Cj}:
\begin{align}
C_j = ((\overline{E}-\underline{E})-\eta_j)^\frac{1}{2} \exp\left(\frac{1}{2} \sum_{k\neq j} \frac{\gamma_k}{\eta_{j,k}}\right).
\end{align}
A simple estimate shows $\|\Psi_j\|_\infty \leq C_j$ (cf. \cite[Lemma 6.1]{CRA89}), and thus $\|\Psi_j\|_{\cD(E)}$ is finite by \eqref{eqn:craigcondn}.

Next, note that $\Psi_j$ is differentiable in the directions $\varphi_k$, with derivative
\begin{align}
\label{eqn:psideriv}
\frac{\partial\Psi_j}{\partial \varphi_k} = \begin{cases}
\frac{1}{\mu_j - \mu_k} \Psi_j \gamma_k \sin(\varphi_k) & k \neq j \\
\frac{1}{2}\left(\frac{1}{\mu_j - \underline{E}} + \frac{1}{\mu_j - \overline{E}} + \sum_{l\neq j} \left( \frac{1}{\mu_j - E_l^+} + \frac{1}{\mu_j - E_l^-} - \frac{2}{\mu_j - \mu_l} \right) \right) \Psi_j \gamma_j \sin(\varphi_j) & k=j
\end{cases}
\end{align}
Another basic estimate shows that
\begin{align*}
\|\Psi(\theta) - \Psi(\varphi)\| \leq \sup_j \left(\sum_k \frac{\gamma_j^\frac{1}{2}}{\gamma_k^\frac{1}{2}} \|\frac{\partial \Psi_j}{\partial \varphi_k}\| \right)\|\theta - \varphi\|.
\end{align*}
Splitting the sum in two, we estimate the sum away from $j$:
\begin{align*}
\sup_j\sum_{k\neq j}\frac{\gamma_j^\frac{1}{2}}{\gamma_k^\frac{1}{2}} \|\frac{\partial \Psi_j}{\partial \varphi_k}\| \leq \sup_j \sum_{k \neq j} \frac{\gamma_j^\frac{1}{2}\gamma_k^\frac{1}{2}}{\eta_{j,k}} C_j.
\end{align*}
This is uniformly bounded by \eqref{eqn:craigcondn}.

To estimate the $j^{th}$ term, note that we can rewrite
\begin{align*}
\frac{1}{\mu_j - E_l^+} + \frac{1}{\mu_j - E_l^-} - \frac{2}{\mu_j - \mu_l} = \frac{\mu_k - E_k^+}{(E_k^- - \mu_j)(\mu_k - \mu_j)} + \frac{\mu_k - E_k^-}{(E_k^+-\mu_j)(\mu_k - \mu_j)}.
\end{align*}
Then we have from \eqref{eqn:psideriv} that
\begin{align*}
\|\frac{\partial \Psi_j}{\partial \phi_j}\| &\leq \sup_j \left( \frac{\gamma_j}{\eta_j} + \sum_{k \neq j} \frac{\gamma_j\gamma_k}{\eta_{j,k}^2} \right) C_j \\
&\leq \sup_j \left( \frac{\gamma_j}{\eta_j} + \left(\sum_{k \neq j} \frac{\gamma_j^\frac{1}{2}\gamma_k^\frac{1}{2}}{\eta_{j,k}}\right)^2 \right) C_j.
\end{align*}
These quantities are all uniformly bounded by \eqref{eqn:craigcondn}.  So indeed, all of the relevant sums are uniformly bounded, and we have that
\begin{align*}
\sup_j \left(\sum_k \frac{\gamma_j^\frac{1}{2}}{\gamma_k^\frac{1}{2}} \|\frac{\partial \Psi_j}{\partial \varphi_k}\| \right) < \infty,
\end{align*}
and the vector field is Lipschitz, as claimed.
\end{proof}

We now show that the vector field $\Psi$ describes the Toda flow on $\cD(E)$:

\begin{prop}
\label{prop:dirichtimederiv}
Suppose $E$ satisfies \eqref{eqn:craigcondn}, and let $(a,b)(t)$ solve \eqref{eqn:todalattice1},\eqref{eqn:todalattice2}, and \eqref{eqn:todalatticeIV} with $(\tilde{a},\tilde{b}) \in \cJ(E)$.  Then, the corresponding function $\varphi(t) \in \cD(E)$ is differentiable in $t$ and obeys
\begin{align}
\label{eqn:dirichtimederiv}
\partial_t \varphi_j(t) = \Psi_j(\varphi(t))
\end{align}
\end{prop}

We prove part of this proposition away from gap edges:

\begin{lem}
\label{lem:dirichtimederiv}
Under the assumptions of Proposition \ref{prop:dirichtimederiv}, at any $t$ such that $\varphi_j(t) \notin \pi \bbZ$, $\varphi_j(t)$ is differentiable and obeys \eqref{eqn:dirichtimederiv}.
\end{lem}

\begin{proof}
The $M$-matrix is analytic in $z \in \bbC \setminus E$.  It has a removable singularity at the Dirichlet eigenvalues, where it is equal to
\begin{align}
\label{eqn:Mmatrixsing}
M(J(t),\mu_j(t)) &= \begin{bmatrix}
m_{-\sigma_j(t)}(\mu_j(t)) & -\frac{\sigma_j(t)}{2a_0(t)}  \\
-\frac{\sigma_j(t)}{2a_0(t)} & 0
\end{bmatrix}.
\end{align}
Recall that $M_{22}$ is exactly the diagonal Green's function $r(J(t),z)$.

At any $t$ such that $\mu_j(t) \in (E_j^-,E_j^+)$, by applying the implicit function theorem to the relation $r(J(t), \mu_j(t)) = 0$, we have that $\mu_j(t)$ is differentiable, and
\begin{align}
\label{eqn:impfunthm}
\dfrac{\partial \mu_j}{\partial t} &= - \dfrac{\partial_t r|_{z = \mu_j}}{\partial_z r|_{z = \mu_j}} \\
\label{eqn:mutimederiv}
&= \frac{1}{2}((E_j^+ - \mu_j)(E_j^- - \mu_j))^{\frac{1}{2}}\Psi_j(\varphi) (\partial_tM_{22}(J(t),z)|_{z = \mu_j}).
\end{align}
By \eqref{eqn:Mmatrixtderiv} and \eqref{eqn:derivfactmat}, we note that
\begin{align*}
\partial_t M_{22}(J(t),z)|_{z=\mu_j} &= 2(B_{21}M_{12} + B_{22}M_{22})|_{z=\mu_j} \\
&= -2\sigma_j(t).
\end{align*}
Thus, we obtain an expression for $\dfrac{\partial \mu_j}{\partial t}$ in terms of the Dirichlet data:
\begin{align}
\label{eqn:mutimederiv2}
\dfrac{\partial \mu_j}{\partial t} &= -\sigma_j(t)((E_j^+ - \mu_j)(E_j^- - \mu_j))^{\frac{1}{2}}\Psi_j(\varphi)
\end{align}
By Proposition \ref{lem:craigequiv}, we likewise find that $\mu_j$ is continuous in $(E_j^-,E_j^+)$.

Next, we claim that $\sigma_j(t)$ is constant while $\mu_j(t) \in (E_j^-,E_j^+)$.  Indeed,
\begin{align}
\label{eqn:sigmatderiv}
\partial_t (M_{12}(J(t),z) + M_{21}(J(t),z)) &= \tr(B)(M_{12} + M_{21}) + \tr(M)(B_{12} + B_{21}).
\end{align}
Evaluating at $\mu_j(t)$ via \eqref{eqn:Mmatrixsing} and noting that $\tr(B) = -\frac{\dot{a}_0}{a_0}$ and that $B$ is anti-symmetric, \eqref{eqn:sigmatderiv} becomes
\begin{align*}
-\frac{\partial_t \sigma_j}{a_0} + \sigma_j \frac{\dot{a}_0}{a_0^2} = \sigma_j \frac{\dot{a}_0}{a_0^2},
\end{align*}
Thus, $\partial_t \sigma_j = 0$ in this situation, and $\sigma_j(t)$ is constant in $t$ while $\mu_j(t) \in (E_j^-,E_j^+)$.

\begin{rem}
Note that, as a result of the $a_0(t)$ factors in the Weyl M-matrix, this aspect of the proof varies slightly from the continuum case in \cite{BDGL15}.  Namely, differentiating the off-diagonal terms picks up an additional summand.  However, the fact that $B$ is no longer traceless saves the day here.
\end{rem}

Thus, from \eqref{eqn:mutophi} and \eqref{eqn:sigma}, we conclude that $\varphi_j(n,t)$ is likewise differentiable in $t$, and
\begin{align*}
\dfrac{\partial \mu_j}{\partial t} &= -\frac{1}{2}(E_j^+ - E_j^-)\sin(\varphi_j)\dfrac{\partial \varphi_j}{\partial t} \\
&= -\sigma_j ((E_j^+ - \mu_j)(E_j^- - \mu_j))^{\frac{1}{2}} \dfrac{\partial \varphi_j}{\partial t}.
\end{align*}
Solving for $\frac{\partial \varphi_j}{\partial t}$ and utilizing \eqref{eqn:mutimederiv2} concludes the proof.
\end{proof}

\begin{lem}[\cite{BDGL15}]
$\varphi(t)$ is continuous.
\end{lem}

\begin{proof}
cf. \cite{BDGL15}, Lemma 3.4.
%
\end{proof}

In particular, $\varphi$ is uniformly continuous on compacts.  We can now prove Proposition \ref{prop:dirichtimederiv}:

\begin{proof}[Proof of Proposition \ref{prop:dirichtimederiv}]
To show differentiability, Lemma \ref{lem:dirichtimederiv} and Proposition \ref{prop:disczeros} reduce the proof to the case where $\varphi_j(t_0) \in \pi\bbZ$.

Suppose $\varphi_j(t_0) \in \pi\bbZ$.  By Proposition \ref{prop:disczeros}, we can shrink $\varepsilon$ such that $\{t\in (t_0-\varepsilon,t_0+\varepsilon): \varphi_j(t) \in \pi\bbZ\} = \{t_0\}$.  The remainder of the proof is straightforward.  By continuity, Lemma \ref{lem:dirichtimederiv}, and the Fundamental Theorem of Calculus, we have
\begin{align*}
\int_{t_0}^{t}\Psi_j(\varphi_j(\tau))d\tau = \varphi_j(t) - \varphi_j(t_0), \; \; t \in (t_0, t_0+\varepsilon)
\end{align*}
(and similarly for $t \in (t_0 - \varepsilon, t_0)$).  Consequently, the derivative of $\varphi_j(t)$ exists at $t_0$, and
\begin{align*}
\partial_t \varphi_j(t)|_{t=t_0} = \Psi_j(\varphi(t_0)).
\end{align*}
This exhausts all cases, and completes the proof.
\end{proof}

\section{Linearization of the Toda Flow}


In this section, we will consider the behavior of the finite-gap solutions $(a,b)^N$ to the Toda lattice and their limit $(a,b)$ with respect to the generalized Abel map of Sodin-Yuditskii \cite{SY97}, and finally prove Theorem \ref{thm:mainthm}.  To prove these facts, we will need to recall additional details from the work of Sodin and Yuditskii \cite{SY95, SY97}.  We assume throughout that $E$ is homogeneous.

For what follows, we assume that the gaps of $E$ are indexed by the positive integers, i.e. $I = \bbN$.  This doesn't reduce generality, because in the case where $I$ is finite, these results are already known \cite{GHMT08}.  Denote
\begin{align*}
E^N = [\underline{E},\overline{E}] \setminus \bigcup_{j \leq N} (E_j^-,E_j^+)
\end{align*}
and denote the corresponding isospectral torus by $\cD(E^N)$ and corresponding trace maps $\cV^N$, vector field $\Psi^N$, and Sodin-Yuditskii translation constant $\alpha^N \in \Gamma^*_{E^N}$.

In their paper, Sodin-Yuditskii introduce functions $\xi_j(z)$ for each gap $(E_j^-,E_j^+)$ of $E$.  The map $\xi_j(z)$ is the solution of the Dirichlet problem on $\bbC \setminus E$ with boundary conditions given by
\begin{align}
\label{eqn:ximap}
\xi_j(z) = \begin{cases}
1 & x \in E, x\geq E_j^+ \\
0 & x\in E, x \leq E_j^-
\end{cases}.
\end{align}
The regularity of $E$ implies that $\xi_j(z)$ is a continuous function on $\hat{\bbC}$ and a harmonic function on $\bbC \setminus E$, with values on $E$ given by \ref{eqn:ximap}.

Let $\pi(\bbC \setminus E)$ be the fundamental group of $\bbC \setminus E$.  It is a free group with the set of generators given by $\{c_j\}_{j\in I}$, where $c_j$ is a counterclockwise simple loop intersecting $\bbR$ at $\underline{E} - 1$ and $(E_j^++E_j^-)/2$.

Following Sodin-Yuditskii, consider the group $\pi^*(\bbC \setminus E)$ of unimodular characters of $\pi(\bbC \setminus E)$, with additive notation for the composition law.  An element $\alpha \in \pi^*(\bbC \setminus E)$ is uniquely determined by its action on loops $c_j$, so we can write $\alpha = \{\alpha_j\}_{j \in I}$, where $\alpha_j = \alpha(c_j) \in \bbT$.  Endow $\pi^*(\bbC \setminus E)$ with the topology dual to the discrete topology on $\pi(\bbC \setminus E)$; one metric which would induce this topology is
\begin{align*}
d(\alpha,\tilde{\alpha}) = \sum_{j \in I} \min(|\alpha_j - \tilde{\alpha}_j|,\gamma_j), \; \alpha, \tilde{\alpha} \in \pi^*(\bbC \setminus E)
\end{align*}
where $\gamma_j = E_j^+ - E_j^-$, as before.  An important realization about this topology is that projections onto the first $N$ coordinates converge uniformly to the identity (since $\sum \gamma_j < \infty$).

Sodin-Yuditskii define the Abel map $\cA : \cD(E) \to \pi^*(\bbC \setminus E)$ by defining its components $\cA_j := \cA_j(c_j)$,
\begin{align}
\label{eqn:abelmap}
\cA_j(\varphi) &= \pi\sum_{k \in I} \sigma_k(\xi_j(\mu_k) - \xi_j(E_k^-)) \mod 2\pi \bbZ
\end{align}
where, as before, we assume $\mu_k, \sigma_k$ are given by $\varphi_k$ as in \eqref{eqn:mutophi}, \eqref{eqn:sigma}.  They then prove the following:
\begin{enumerate}
\item For each $j$, the sum in \eqref{eqn:abelmap} converges absolutely and uniformly in $\varphi \in \cD(E)$; in particular, the map $\cA$ is well-defined.
\item $\cA$ is a homeomorphism between $\cD(E)$ and $\pi^*(\bbC \setminus E)$ linearizing the translation flow: there exists $\alpha_E \in \pi^*(\bbC \setminus E)$ such that
\begin{align*}
\cA(\varphi(n)) = \cA(\varphi(0)) + n\alpha_E.
\end{align*}
This, in turn, uniquely defines $\varphi(n)$ given $\varphi(0)$; thus, $\cB:\cJ(E) \to \cD(E)$ is a bijection (in fact, a homeomorphism).
\item Denote by $\cA^N:\cD(E^N) \to \pi^*(\bbC \setminus E^N)$ the Abel map for $E^N$.  Project $\cD(E)$ to $\cD(E^N)$ by truncation and embed $\pi^*(\bbC \setminus E^N)$ into $\pi^*(\bbC \setminus E)$ by assuming $\alpha^N(c_j) = 0$ for $j > N$.  With these conventions, consider $\cA^N$ as a map from $\cD(E)$ to $\pi^*(\bbC \setminus E)$.  Then, $\cA^N \to \cA$ as $N \to \infty$, uniformly on $\cD(E)$.
\end{enumerate}
This extended Abel map $\cA$ was introduced to generalize the notion of Jacobi inversion, which exists in the finite-gap setting.  We will use this language to reinterpret the Toda and translation flows on the Dirichlet data:

\begin{lem}
\label{lem:existence}
Suppose $E$ satisfies \eqref{eqn:craigcondn}, and let $f \in \cD(E)$.  Then there exists a function $\varphi : \bbZ \times \bbR \to \cD(E)$ such that $\varphi(0,0) = f$, and
\begin{align}
\label{eqn:divflow1}
\varphi(n+1,t) &= \cA^{-1}(\cA(\varphi(n,t)) + \alpha), \\
\label{eqn:divflow2}
\partial_t\varphi(n,t) &= \Psi(\varphi(n,t)).
\end{align}
If we define $(a,b) : \bbZ \times \bbR \to \bbR^2$ by
\begin{align*}
(a,b) = \cV \circ \varphi
\end{align*}
then the function $(a,b)$ satisfies \eqref{eqn:todalattice1},\eqref{eqn:todalattice2}, and \eqref{eqn:todalatticeIV}.  Moreover, for each $t \in \bbR$, we have $(a,b)(t) \in \cJ(E)$, and $\cB((a,b)(t)) = \varphi(0,t)$.
\end{lem}

\begin{proof}
This is an immediate corollary of the existence and uniqueness of solutions to the Toda IVP \cite[Theorem 12.6]{TES00}, the preservation of the reflection coefficients under the Toda flow \cite[Proposition 1.4]{REM15}, and the work of Sodin-Yuditskii \cite[Theorem C]{SY97}.
\end{proof}

We will prove Theorem \ref{thm:mainthm} via approximation, and begin by constructing our finite-gap approximants.  Starting from an element $f \in \cD(E)$, observe the functions $\varphi^N: \bbZ \times \bbR \to \cD(E^N)$ solving
\begin{align}
\label{eqn:findirichdiffeqn2}
\partial_t \varphi^N(n,t) &= \Psi^N(\varphi^N(n,t)), \\
\varphi^N(n+1,t) &= (\cA^N)^{-1}(\cA^N(\varphi^N(n,t)) + \alpha^N)
\end{align}
obeying initial condition
\begin{align*}
\varphi_j^N(0,0) = f_j, \; j \leq N.
\end{align*}
This uniquely determines the function $\varphi^N$, by results of \cite{GHMT08} (although condition \eqref{eqn:divflow1} appears here in a different, but equivalent, form).

Furthermore, by \cite{GHMT08}, the function $(a,b)^N = \cV^N\circ \varphi^N$ has the following properties:
\begin{enumerate}
\item For every $t$, $J^N(t) := (a,b)^N(t)$ is almost periodic and the Jacobi operator $J^N(t)$ has spectrum $E^N$.
\item $(a,b)^N$ satisfies the Toda lattice \eqref{eqn:todalattice1},\eqref{eqn:todalattice2}.
\item The Dirichlet data of $(a,b)^N$ are $\varphi^N$.
\end{enumerate}

Our ultimate goal will be to compare trajectories on finite-gap approximants to our isospectral torus.  To do so, we employ a lemma from \cite{BDGL15}, which we quote here for convenience:

\begin{lem}\cite[Lemma 4.4]{BDGL15}
\label{lem:vectfieldbd}
Assume that $U$, $\tilde{U}$ are Lipschitz vector fields on $\cD(E)$, with Lipschitz constants less than or equal to $L$.  Consider solutions $\phi, \tilde{\phi}: \bbR \to \cD(E)$ of $\partial_t \phi = U(\phi)$, $\partial_t \tilde{\phi} = \tilde{U}(\tilde{\phi})$.  Then
\begin{align}
\label{eqn:vectfieldbd}
\|\phi(t) - \tilde{\phi}(t)\|_{\cD(E)} \leq 2 \left( \|\phi(0) - \tilde{\phi}(0)\|_{\cD(E)} + C\right) e^{m|t|},
\end{align}
where $m = 2L\log(2)$ and
\begin{align*}
C = \frac{1}{L}\sup_{j \in I} \gamma_j^\frac{1}{2} \min(2\pi, \|U_j - \tilde{U}_j\|).
\end{align*}
\end{lem}

To apply this lemma to compare trajectories on different isospectral tori, we lift all the fields and solutions to $\cD(E)$ as follows.

There is a natural projection $\pi^N : \cD(E) \to \cD(E^N)$ given by
\begin{align*}
\pi(\phi)_j = \phi_j, \; j \leq N.
\end{align*}
Lift the trace formulas to $\tilde{\cV}^N : \cD(E) \to \bbR^2$ by
\begin{align*}
\tilde{\cV}^N &= \cV^N \circ \pi.
\end{align*}
Introduce the vector fields $\tilde{\Psi}^N$ on $\cD(E)$ by
\begin{align*}
\tilde{\Psi}^N_j(\phi) &= \begin{cases}
\Psi^N_j(\pi(\phi)) & j \leq N \\
\Psi_j(\phi) & j > N.
\end{cases}
\end{align*}
To define $\tilde{\varphi}^N : \bbZ \times \bbR \to \cD(E)$, set the initial value
\begin{align*}
\tilde{\varphi}^N(0,0) = f,
\end{align*}
and flow via
\begin{align*}
\partial_t \tilde{\varphi}^N(0,t) &= \tilde{\Psi}^N(\tilde{\varphi}^N(0,t)), \; \forall t \in \bbR, \\
\tilde{\varphi}_j^N(n,t) &= \begin{cases}
((\cA^N)^{-1}(\cA^N(\tilde{\varphi}^N(0,t)) + n\alpha^N))_j & j \leq N \\
((\cA)^{-1}(\cA(\tilde{\varphi}^N(0,t)) + n\alpha))_j & j > N
\end{cases}
\end{align*}
With these definitions, one can check that
\begin{align*}
\varphi^N = \pi \circ \tilde{\varphi}^N.
\end{align*}
In particular, note that $\cV^N(\varphi^N) = \tilde{\cV}^N(\tilde{\varphi}^N)$.  Here, it is important that we established that $\Psi$ (and, consequently, $\tilde{\Psi}^N$) is Lipschitz to guarantee existence and uniqueness of the functions $\tilde{\varphi}^N$.

Finally, introduce $\varphi : \bbZ \times \bbR \to \cD(E)$ by
\begin{align*}
\varphi(0,0) &= f, \\
\varphi(n+1,0) &= \cA^{-1}(\cA(\varphi(n,0)) + \alpha), \; \forall n \in \bbZ \\
\partial_t\varphi(n,t) &= \Psi(\varphi(n,t)), \; \forall (n,t) \in \bbZ \times \bbR.
\end{align*}

It will become important to consider the convergence of the inverses of the finite-gap Abel maps $\cA^N$.  By an abuse of notation, we will denote by $(\cA^N)^{-1} : \pi^*(\bbC\setminus E) \to \cD(E)$ the lift of the proper inverse of $\cA^N$, such that
\begin{align*}
((\cA^N)^{-1}(\cA^N)(\phi))_j = \begin{cases}
\phi_j & j \leq N \\
0 & j > N
\end{cases}.
\end{align*}
Under these assumptions, it is a straightforward exercise to prove
\begin{lem}
\label{lem:unifconvinv}
$(\cA^N)^{-1}$ converges uniformly to $\cA^{-1}$.
\end{lem}
\begin{proof}
By \cite[Equation (4.1.2)]{SY95}, convergence of $\cA^N$ to $\cA$ is uniform, and by the definition of the metric on $\pi^*(\bbC \setminus E)$ and compactness, lifted projections $\tilde{\alpha}^N$ of a $\alpha \in \pi^*(\bbC \setminus E)$ likewise converge uniformly to $\alpha$.  Since $\pi^*(\bbC \setminus E)$ is compact, $\cA^{-1}$ is uniformly continuous.

Let $\varepsilon > 0$ be given, and find the corresponding uniform $\delta$ for $\cA^{-1}$.  There exists an integer $N_0$ such that for all $N > N_0$, for all $\alpha \in \pi^*(\bbC \setminus E)$, $d(\alpha,\tilde{\alpha}^N) < \delta/2$ and $\|\cA^N - \cA\|_\infty < \delta/2$.  Then
\begin{align*}
\|\cA(\cA^{-1}(\alpha)) - \cA((\cA^N)^{-1}(\alpha))\| &\leq \|\alpha - \tilde{\alpha}^N\|+ \|\cA^N - \cA\|_\infty \\
&< \delta/2 + \delta/2 = \delta.
\end{align*}
Consequently, by uniform continuity of $\cA^{-1}$, for $N > N_0$
\begin{align*}
\|\cA^{-1}(\cA(\cA^{-1}(\alpha))) - \cA^{-1}(\cA((\cA^N)^{-1}(\alpha)))\| &= \|\cA^{-1}(\alpha) - (\cA^N)^{-1}(\alpha)\| \\
&< \varepsilon,
\end{align*}
i.e. $(\cA^N)^{-1}$ converges uniformly to $\cA^{-1}$.
\end{proof}

\begin{lem}
\label{lem:finsolnapprox}
Suppose $E$ satisfies \eqref{eqn:craigcondn}, and let $K > 0$.  Then there exists $m > 0$ and constants $K_N$ such that $\lim_{N \to \infty} K_N = 0$ and, for all $t \in \bbR$,
\begin{align}
\label{eqn:finsolnapprox}
\sup_{|n|\leq K} \|\tilde{\varphi}^N(n,t) - \varphi(n,t)\| \leq K_Ne^{m|t|}.
\end{align}
\end{lem}

\begin{proof}By our assumptions on $E$ and Lemma \ref{lem:craigequiv}, $\Psi$ is a Lipschitz vector field.  The proof of this fact gives upper bounds for the Lipschitz constants in terms of gap sizes and distances, so it also applies to $\tilde{\Psi}^N$, giving uniform Lipschitz estimates in $N$.  Denote by $L$ such an upper bound on the Lipschitz constant which works for all $\tilde{\Psi}^N$ and $\Psi$.

In particular, under these assumptions, the values $\tilde{\varphi}^N(0,t)$ and $\varphi(0,t)$ defined above are all uniquely determined by existence and uniqueness theorems for differential equations, and consequently $\tilde{\varphi}^N(n,t)$ and $\varphi(n,t)$ are uniquely determined for each $n$.

One can show (cf. \cite[Lemma 4.5]{BDGL15}) that, for $j \leq N$, we have
\begin{align*}
\|\Psi_j - \tilde{\Psi}_j^N\| \leq 2(C_j- C_{j,N})
\end{align*}
where $C_j$ is defined as in \ref{eqn:Cj} and
\begin{align*}
C_{j,N} = (\overline{E} -\underline{E})^\frac{1}{2}\exp\left(\frac{1}{2}\sum_{\substack{l\leq N \\ l\neq j}} \frac{\gamma_l}{\eta_{j,l}}\right).
\end{align*}
Consequently, it follows that
\begin{align*}
\lim_{N \to \infty}\|\Psi_j - \tilde{\Psi}_j^N\|= 0.
\end{align*}

Let $m = 2L\log(2)$ and let
\begin{align*}
\tilde{K}_N = \frac{1}{L} \sup_{j \leq N} \gamma_j^\frac{1}{2} \max(2\pi, \|\Psi_j - \tilde{\Psi}_j^N\|).
\end{align*}
Since $\gamma_j \to 0$, it follows that $\tilde{K}_N \to 0$ as $N \to \infty$.

Applying Lemma \ref{lem:vectfieldbd} to compare trajectories of the vector fields $\tilde{\Psi}^N$ and $\Psi$ with initial condition $f$, we find
\begin{align*}
\|\tilde{\varphi}^N(n,t) - \varphi(n,t)\| \leq (2\|\tilde{\varphi}^N(n,0) - \varphi(n,0)\|+ \tilde{K}_N)e^{m|t|}.
\end{align*}

To verify our claim, it suffices to show $\sup_{|n|\leq K} \|\tilde{\varphi}^N(n,0) - \varphi(n,0)\| \to 0$ as $N \to \infty$.  We analyze by entry.  There are two cases.  First, if $j > N$, we have that $|\tilde{\varphi}_j^N(n,0) - \varphi_j(n,0)|=0$.  Otherwise, if $j \leq N$,
\begin{align*}
|\tilde{\varphi}_j^N(n,0) - \varphi_j(n,0)| &= |((\cA^N)^{-1}(\cA^N(f) + n\alpha^N))_j - \cA^{-1}(\cA(f) + n\alpha)_j| \\
&\leq |((\cA^N)^{-1}(\cA^N(f) + n\alpha^N))_j - \cA^{-1}(\cA^N(f) + n\alpha^N))_j| \\ & \; \; \; \; \; \; \; \; \; \; \; \; \; \; + |\cA^{-1}(\cA^N(f) + n\alpha^N))_j - \cA^{-1}(\cA(f) + n\alpha)_j|.
\end{align*}
By Lemma \ref{lem:unifconvinv} and the uniform convergence  of $\cA^N$ and $\alpha^N$ to $\cA$ and $\alpha$ respectively \cite{SY95}, it follows that these terms go to zero uniformly in $N$ for each $|n| \leq K$.

Taking $K_N = 2\max(\tilde{K}_N, 2\sup_{|n|\leq K} \|\tilde{\varphi}^N(n,0) - \varphi(n,0)\|)$, the claim is proved.
\end{proof}

In the finite-gap setting, it is known that Jacobi inversion linearizes both translation and Toda flows \cite[Theorem 1.41]{GHMT08}.  Therefore,
\begin{align}
\label{eqn:fingaplinearize}
\cA^N(\varphi^N(n,t)) = \cA^N(\varphi^N(0,0)) + n\alpha^N + \z^Nt
\end{align}
for some $\alpha^N, \z^N \in \bbR^N$.

Define the map
\begin{align*}
\cM := \cB^{-1} \circ \cA^{-1} : \pi^*(\bbC \setminus E) \to \cJ(E).
\end{align*}
If $E$ satisfies \eqref{eqn:craigcondn}, then the map $\cM$ is a homeomorphism by the considerations of Sodin and Yuditskii \cite{SY97}.  We will use this map and proceed as in \cite{BDGL15} to prove our main theorem.

\begin{proof}[Proof of Theorem \ref{thm:mainthm}]
Since $J_0$ is almost periodic and $\sigma(J_0) = \sigma_{ac}(J_0) = E$, a result of Remling \cite[Theorem 1.4]{REM11} implies that $J_0 \in \cJ(E)$.  By \cite[Theorem 12.6]{TES00}, there is a unique solution $J(t)$ to \eqref{eqn:todalattice1}, \eqref{eqn:todalattice2}, \eqref{eqn:todalatticeIV}, and by \cite[Proposition 1.4]{REM15}, $J(t) \in \cJ(E)$ for each $t \in \bbR$.  Consequently, the solution $J(t)$ is uniformly bounded in time.  By Section 3 and Sodin-Yuditskii \cite[Theorem C]{SY97}, the Dirichlet data $\varphi(n,t) := \cB(S^nJ(t))$ obey
\begin{align}
\label{eqn:dirichdiffeq}
\varphi(n+1,t) = \cA^{-1}(\cA(\varphi(n,t)) + \alpha), \;
\partial_t\varphi(n,t) = \Psi(\varphi(n,t)).
\end{align}

We now recall the functions $\tilde{\varphi}^N(n,t)$ introduced above.  Since $\tilde{\varphi}^N \to \varphi$ uniformly on compacts and $\cA^N \to \cA$ uniformly, we can conclude that $\cA^N(\tilde{\varphi}^N(n,t))$ converges uniformly on compacts to $\cA(\varphi(n,t))$.  Note that, by definition, $\cA_j^N(\tilde{\varphi}^N(n,t)) = \cA_j^N(\varphi^N(n,t))$ for $j \leq N$.  Taking the $j^{th}$ component of \eqref{eqn:fingaplinearize}, it follows from uniform convergence that the limits
\begin{align*}
\alpha_j &= \lim_{N \to \infty} \alpha_j^N, \\
\z_j &= \lim_{N \to \infty} \z_j^N
\end{align*}
exist, and
\begin{align}
\label{eqn:linearize}
\cA_j(\varphi(n,t)) = \cA_j(\varphi(0,0)) + n\alpha_j + \z_jt.
\end{align}
In particular, $\cM^{-1}((a,b)(t)) = \cM^{-1}(J_0) + \z t$.  This proves the time almost periodicity of solutions.

Finally, the spatial almost periodicity of solutions and the equivalence of frequency modules follow from the considerations of Sodin-Yuditskii \cite{SY97}.
\end{proof}

\appendix

\section{Vanishing Lyapunov Exponents in the Sodin-Yuditskii Regime}

In the context of spectral homogeneity, it is a sort of folklore that we have an equivalence of the equilibrium and density of states measures (denoted $d\rho_E$ and $dk$, respectively), and, in particular, everywhere zero Lyapunov exponents for reflectionless Jacobi operators.  We offer a formal write-up of this fact, perhaps initially observed in Eremenko-Yuditskii \cite{EY12}, here.  Throughout, we assume our fixed compact set $E$ is homogeneous.

The equivalence of equilibrium and density of states measures in the Sodin-Yuditskii regime is most quickly seen via a theorem of Simon:

\begin{thm}[Theorem 1.15, \cite{SIM07}]
\label{thm:simon07}
For an ergodic family of Jacobi operators $J(\omega)$, let $E = \supp(dk)$ and $C(E)$ be the logarithmic capacity of $E$.  If $\lim_{n\to\infty}(a_1a_2...a_n)^{\frac{1}{n}} = C(E)$, then $dk(x) = d\rho_E(x)$.
\end{thm}
\noindent Here, we have denoted by $C(E)$ the potential-theoretic capacity of $E$.  To apply this theorem in our situation, we once more appeal to the work of Sodin and Yuditskii.

The Hardy spaces of character automorphic functions on the disk $H^2(\omega)$, $\omega \in \pi^*(\bbC \setminus E)$, defined in the previous section, are Hilbert spaces with reproducing kernels $k^\omega : \bbD \to \bbC$.  In \cite[Theorem F]{SY97}, these kernels are used to describe continuous trace formulas:

\begin{align*}
\cQ(\varphi) :&= \frac{1}{2}\left(\underline{E} + \overline{E} + \sum_{j \in I} (E_j^- + E_j^+ - 2\mu_j)\right) \\
\cP(\varphi) :&= C(E)\frac{k^{\cA(\varphi)+\alpha_E}(0)}{k^{\cA(\varphi)}(0)} \\
&= C(E)\exp\left(-\frac{1}{2}\sum_{j \in I}(g_{\bbC \setminus E}(\mu_j^+) - g_{\bbC \setminus E}(\mu_j) + (\sigma_j - \sigma_j^+)g_{\bbC \setminus E}(c_j))\right)
\end{align*}
Here, $(\mu_j^+,\sigma_j^+)$ correspond to $\varphi^+ := \cA^{-1}(\cA(\varphi)+\alpha_E)$, where $\alpha_E \in \pi^*(\bbC \setminus E)$ is completely determined by $E$.  It is shown that, when $E$ is homogeneous, a Jacobi operator can be recovered exactly from its Dirichlet data via these trace formulas:
\begin{thm}[Theorem F, \cite{SY97}]
\label{t:syform}
Let $E$ be homogeneous, let $J = (a,b) \in \cJ(E)$, and let $\cB(J) = \varphi$.  Then
\begin{align*}
a_n &= \cP(\cA^{-1}(\cA(\varphi)+n\alpha_E)) \\
b_n &= \cQ(\cA^{-1}(\cA(\varphi)+n\alpha_E)).
\end{align*}
\end{thm}

In particular, the shift action on $\cJ(E)$ conjugates to translation by a constant vector, which, by definition, preserves the Haar measure on $\cJ(E)$.  By continuity of the above maps, it follows that every element of $\cJ(E)$ is almost periodic, and, in fact, ergodic.  With this theorem and the aforementioned theorem of Simon, we find the following

\begin{prop}
\label{prop:zerole}
For any $J \in \cJ(E)$, we have $dk(x) = d\rho_E(x)$.
\end{prop}

This follows quite easily from the following Lemma:


\begin{lem}[\cite{VY14}, Corollary 2.4]
\label{lem:kernineq}
For all $\omega \in \pi^*(\bbC\setminus E)$, the reproducing kernels $k^{\omega}$ satisfy
\begin{align}
\label{eqn:kernineq}
0 < |\theta(0)| \leq k^{\omega}(0) \leq 1
\end{align}
\end{lem}

Applying the Sodin-Yuditskii characterization of the off-diagonal elements $a_n$ above, the proof of the above Proposition is now a simple calculation:

\begin{proof}[Proof of Proposition \ref{prop:zerole}]
Note now that, by Theorem \ref{t:syform}, we have
\begin{align}
\label{eqn:geommean}
(a_1a_2...a_n)^{\frac{1}{n}} = C(E)\Big(\frac{k^{\omega+(n+1)\alpha}}{k^{\omega+\alpha}}(0)\Big)^{\frac{1}{n}}
\end{align}
By Lemma \ref{lem:kernineq}, $\frac{k^{\omega+(n+1)\alpha}}{k^{\omega+\alpha}}(0)$ is uniformly bounded away from zero and infinity.  Taking limits in (\ref{eqn:geommean}) and applying Theorem \ref{thm:simon07} yields the claimed result.
\end{proof}

Consequently, the Lyapunov exponent vanishes throughout the spectrum:

\begin{thm}
Suppose $E$ is homogeneous, and let $J \in \cJ(E)$.  Then the shift cocycle $S(1;J,z)$ has Lyapunov exponent $L(x)$ zero for all $x \in E$.
\end{thm}

\begin{proof}
Since $E$ is homogeneous, it is regular for potential theory, and Green's function $G(z)$ is everywhere continuous and vanishes throughout the spectrum.  Thus, for all $x \in E$,
\begin{align}
\label{eqn:zerogreens}
G(x) = \int_E \log|t-x|d\rho_E(t) - \log(C(E)) = 0
\end{align}
By equality of the DOS and equilibrium measures, \eqref{eqn:zerogreens}, and the Thouless formula, for all $x \in E$,
\begin{align*}
L(x) = \int_E\log|t-x|dk(t) - \log(C(E)) = 0.
\end{align*}
\end{proof}

\bibliographystyle{plain}
\def\cprime{$'$}

\end{document}